\documentclass[11pt,reqno]{amsart}

\newtheorem{proposition}{Proposition}[section]

\newtheorem{corollary}[proposition]{Corollary}
\newtheorem{theorem}[proposition]{Theorem}

\theoremstyle{definition}
\newtheorem{definition}[proposition]{Definition}
\newtheorem{example}[proposition]{Example}
\newtheorem{examples}[proposition]{Examples}
\newtheorem{remark}[proposition]{Remark}

\newcommand{\thlabel}[1]{\label{th:#1}}
\newcommand{\thref}[1]{Theorem~\ref{th:#1}}
\newcommand{\selabel}[1]{\label{se:#1}}
\newcommand{\seref}[1]{Section~\ref{se:#1}}

\newcommand{\prlabel}[1]{\label{pr:#1}}
\newcommand{\prref}[1]{Proposition~\ref{pr:#1}}
\newcommand{\colabel}[1]{\label{co:#1}}
\newcommand{\coref}[1]{Corollary~\ref{co:#1}}

\newcommand{\delabel}[1]{\label{de:#1}}
\newcommand{\deref}[1]{Definition~\ref{de:#1}}
\newcommand{\eqlabel}[1]{\label{eq:#1}}
\newcommand{\equref}[1]{(\ref{eq:#1})}

\def\ot{\otimes}

\def\NN{{\mathbb N}}
\def\CC{{\mathbb C}}

\newcommand{\Cc}{\mathcal{C}}

\newcommand{\Ll}{\mathcal{L}}

\newcommand{\Mm}{\mathcal{M}}

\def\*C{{}^*\hspace*{-1pt}{\Cc}}
\def\text#1{{\rm {\rm #1}}}

\input xy
\xyoption {all} \CompileMatrices

\usepackage{amssymb}
\usepackage{color,amssymb,graphicx,amscd,amsmath}
\usepackage[colorlinks,urlcolor=blue,linkcolor=blue,citecolor=blue]{hyperref}
\begin{document}

\title[Universal modules]
{Functors between representation categories. Universal modules}

\author{A. L. Agore}
\address{Simion Stoilow Institute of Mathematics of the Romanian Academy, P.O. Box 1-764, 014700 Bucharest, Romania}
\address{Vrije Universiteit Brussel, Pleinlaan 2, B-1050 Brussels, Belgium}
\email{ana.agore@gmail.com, ana.agore@vub.be}

\thanks{This work was supported by a grant of Romanian Ministry of Research, Innovation and Digitization, CNCS/CCCDI --
UEFISCDI, project number PN-III-P4-ID-PCE-2020-0458, within PNCDI III}


%\maketitle

\maketitle

\begin{abstract}
Let $\mathfrak{g}$ and $\mathfrak{h}$ be two Lie algebras with $\mathfrak{h}$ finite dimensional and consider ${\mathcal A} = {\mathcal A} (\mathfrak{h}, \, \mathfrak{g})$ to be the corresponding universal algebra as introduced in \cite{am20}. Given an ${\mathcal A}$-module $U$ and a Lie $\mathfrak{h}$-module $V$ we show that $U \ot V$ can be naturally endowed with a Lie $\mathfrak{g}$-module structure. This gives rise to a functor between the category of Lie $\mathfrak{h}$-modules and the category of Lie $\mathfrak{g}$-modules and, respectively, to a functor between the category of ${\mathcal A}$-modules and the category of Lie $\mathfrak{g}$-modules. Under some finite dimensionality assumptions, we prove that the two functors admit left adjoints which leads to the construction of universal ${\mathcal A}$-modules and universal Lie $\mathfrak{h}$-modules as the representation theoretic counterparts of Manin-Tambara's universal coacting objects \cite{Manin, Tambara}. 
\end{abstract}

\section*{Introduction}
The universal coacting bialgebra/Hopf algebra on a finite dimensional (graded) associative algebra originates in the work of Yu. I. Manin (\cite{Manin}). The importance of this construction became obvious mostly due to its interaction with non-commutative geometry where it is seen as some sort of symmetry group (see \cite{theo} for more details on this view point). The non-graded version of this construction appeared a few years later in a paper by D. Tambara (\cite{Tambara}). However, as remarked in \cite{Tambara}, the universal coacting bialgebra is in fact the dual of the so-called universal measuring bialgebra introduced by M.E. Sweedler in \cite{Sw}. We should note that, unlike Manin-Tambara's construction, Sweedler's universal measuring bialgebra/Hopf algebra exists even in the infinite-dimensional case.

In recent years, universal (co)acting objects have been considered in various settings and for different purposes. For instance, \cite{cris} extends Sweedler's construction to monoids in a braided monoidal category. On the other hand, the Manin-Tambara construction was introduced in the setting of Poisson algebras (\cite{ana2021}), finite index-subfactors (\cite{alex1}), superpotential algebras (\cite{alex2}), polynomial algebras (\cite{Taft}), bialgebroids (\cite{ard}) or Lie/Leibniz algebras (\cite{am20}). The corresponding universal coacting bialgebras/Hopf algebras, which in certain cases carry some extra structure (e.g. a Poisson Hopf algebra structure as in \cite{ana2021}), seem to play a prominent role in solving other seemingly unrelated problems such as the classification of gradings on various kinds of algebras (\cite{am20, mil}), the description of the automorphisms group of certain algebraic structures (\cite{am20}) and even in quantum Galois theory (\cite{alex1}). Another related universal (co)acting construction was considered in \cite{AGV1} as the Hopf algebraic analogue of the universal group of a grading and its connections to the problem of classifying Hopf algebra coactions have been highlighted.

One of the most general constructions of universal (co)acting bialgebras/Hopf algebras, performed in the setting of $\Omega$-algebras, was introduced in \cite{AGV2} together with generalized duality results. Necessary and sufficient conditions for the existence of universal coacting bialgebras/Hopf algebras are provided, explaining in this general setting the need for assuming finite dimensionality in both Manin and Tambara's papers.

It is worth to point out that both Sweedler and Manin-Tambara's constructions have a categorical interpretation. More precisely, for Tambara's construction one considers the left adjoint, say $a(A, - )$, of the tensor product endofunctor $A \ot - $ on the category of $k$-algebras, where $A$ is a finite dimensional associative algebra. Tambara's universal coacting bialgebra is precisely $a(A, A)$ which turns out to be naturally endowed with a bialgebra structure. Similarly, for an arbitrary associative algebra $A$, it can be proved that the contravariant functor ${\rm Hom}(-,\, A)$ taking coalgebras to (convolution) algebras has a right adjoint which hereafter we denote by $M(A,\,-)$. As before, Sweedler's universal measuring bialgebra is exactly $M(A,\,A)$ which again has a bialgebra structure.  

In this paper we deal with the representation theoretic version of Manin-Tambara's construction in the Lie algebra setting. Our approach is a categorical one. More precisely, given two fixed Lie algebras $\mathfrak{g}$ and $\mathfrak{h}$, with $\mathfrak{h}$ finite dimensional, and the corresponding universal algebra ${\mathcal A} = {\mathcal A} (\mathfrak{h}, \, \mathfrak{g})$ as defined in \cite{am20}, we first show that the tensor product between an ${\mathcal A}$-module $U$ and a Lie $\mathfrak{h}$-module $V$ can be endowed with a Lie $\mathfrak{g}$-module structure (\thref{tensor_f}). As a consequence, we are able to construct two ''tensor product'' functors between the categories of Lie modules over $\mathfrak{h}$ and $\mathfrak{g}$ and respectively between the category of ${\mathcal A}$-modules and the category of Lie $\mathfrak{g}$-modules. Under the appropriate  finite dimensionality assumptions, the two functors mentioned above are proved to admit left adjoints. These left adjoints are given precisely by what we have called the universal Lie $\mathfrak{h}$-module and the universal ${\mathcal A}$-module, respectively. The two universal modules are introduced in a constructive manner in \thref{exist1} and \thref{exist2}. These are the counterparts for Lie and associative representations of Manin-Tambara's constructions.

Furthermore, the two aforementioned pairs of adjoint functors allow us to travel both ways between the representation categories of different algebraic structures, such as Lie and associative algebras, and to transfer certain properties which are usually preserved by left/right adjoints.

\section{Preliminaries}\selabel{prel}

This section will be used mostly as an opportunity to fix some notation and to provide certain useful references. Let us start with a few words on notation.
\subsection{Notational conventions}
All vector spaces, (bi)linear maps, unadorned tensor products, Lie or associative
algebras, bialgebras and so on are over an arbitrary commutative field $k$. All (co)associative (co)algebras are assumed to be (co)unital. The notation employed for coalgebras is standard: $\Delta$ stands for the comultiplication and $\varepsilon$ for the counit. We use Sweedler's notation with implied summation for both coalgebras (resp. bialgebras), as in $\Delta(c) = c_{(1)} \ot c_{2}$, and for comodule structures: a right $C$-comodule structure $\rho$ on a vector space $V$ will be denoted by $\rho(v) = v_{(0)} \ot v_{(1)}$. When we need to be precise, the structures involved will be adorned. $\delta_{ij}$ denotes Kronecker's symbol while ${\rm Id}_{X}$ stands for the identity map on the set $X$.

In the sequel, $k [X_{si} \, | s = 1, \cdots, n, \, i\in I]$
denotes the usual polynomial algebra on variables $X_{si}$.
We shall denote by ${\rm Lie}_k$ and ${\rm
ComAlg}_k$ the categories of Lie and commutative associative algebras, respectively. Given an associative algebra $A$ and a Lie algebra $\mathfrak{g}$ we denote by ${}_{A} \Mm$ and ${}_{ \mathfrak{g}}\Ll\Mm$  the categories of
left $A$-modules and left Lie $ \mathfrak{g}$-modules, respectively. Recall that a (left) Lie $\mathfrak{g}$-module is a vector space $V$ equipped with a bilinear map $ \rightharpoonup \colon \mathfrak{g} \times V \to V$ such that for all $x$, $y \in \mathfrak{g}$ and $v \in V$ we have:
\begin{eqnarray*}
[x,\,y]  \rightharpoonup v = x  \rightharpoonup (y  \rightharpoonup v) - y  \rightharpoonup(x  \rightharpoonup v).
\end{eqnarray*}

Throughout the paper, $\mathfrak{g}$ and $\mathfrak{h}$ denote two arbitrary Lie algebras with $ \mathfrak{h}$ finite dimensional, while $\{f_i \, | \, i \in I\}$ and $\{e_1, \cdots, e_n\}$ will be two fixed basis in $\mathfrak{g}$ and $\mathfrak{h}$, respectively. Furthermore, in what follows we consider $\{\tau_{i, j}^s \, | \, i, j, s = 1, \cdots, n \}$ to be the
structure constants of $\mathfrak{h}$, i.e. for any $i$, $j = 1,
\cdots, n$ we have:
\begin{equation}\eqlabel{const1}
\left[e_i, \, e_j \right]_{\mathfrak{h}} = \sum_{s=1}^n \,
\tau_{i, j}^s \, e_s.
\end{equation}
Similarly, for any $i$, $j\in I$, let $B_{i,j} \subseteq I$ be a finite subset of
$I$ such that:
\begin{equation}\eqlabel{const3.4}
\left[f_i, \, f_j \right]_{\mathfrak{g}} = \sum_{u \in B_{i, j}}
\, \beta_{i, j}^u \, f_{u},
\end{equation}
for some scalars $\beta_{i, j}^u \in k$.

\subsection{The universal algebra of $\mathfrak{h}$ and $\mathfrak{g}$}
We recall briefly, for further use, the construction of the universal commutative algebra ${\mathcal A} (\mathfrak{h}, \, \mathfrak{g})$ of two given Lie algebras $\mathfrak{h}$ and $\mathfrak{g}$ (recall that $\mathfrak{h}$ is always assumed to be finite dimensional). 
It was first introduced in \cite{am20} in the more general setting of Leibniz algebras as the counterpart of Tambara's construction (\cite{Tambara}). We restrict here to the Lie algebra version of the construction which can be summarized as follows. We have: 
\begin{equation}\eqlabel{alguniv}
{\mathcal A} (\mathfrak{h}, \, \mathfrak{g}) :=  k [X_{si} \, | s
= 1, \cdots, n, \, i\in I] / \mathcal{J}
\end{equation}
where $ \mathcal{J}$ is the ideal generated by all polynomials of the form
\begin{equation}\eqlabel{poluniv}
P_{(a, i, j)} ^{(\mathfrak{h}, \, \mathfrak{g})} := \sum_{u \in
B_{i, j}} \, \beta_{i, j}^u \, X_{au} - \sum_{s, t = 1}^n \,
\tau_{s, t}^a \, X_{si} X_{tj}, \quad {\rm for}\,\, {\rm all}\,\, a = 1, \cdots, n\,\, {\rm and}\,\, i,\, j\in I.
\end{equation}
When working in the universal algebra ${\mathcal A} (\mathfrak{h}, \,
\mathfrak{g})$, we denote by $x_{si} := \widehat{X_{si}}$ the class of ${X_{si}}$. Consequently, the following relations hold in ${\mathcal A} (\mathfrak{h}, \,
\mathfrak{g})$:
\begin{equation}\eqlabel{relatii}
\sum_{u \in B_{i, j}} \, \beta_{i, j}^u \, x_{au} = \sum_{s, t =
1}^n \, \tau_{s, t}^a \, x_{si} x_{tj}, \quad {\rm for}\,\, {\rm all}\,\, a = 1, \cdots, n,\,\, {\rm and} \,\, i,\, j\in I.
\end{equation}

When the (finite dimensional) Lie algebra $\mathfrak{h}$ is fixed, the universal algebra construction gives rise to a functor ${\mathcal A} (\mathfrak{h},-) \colon {\rm Lie}_k \to {\rm ComAlg}_k$ which turns out to be the left adjoint of the tensor product $\mathfrak{h} \ot - \colon {\rm ComAlg}_k \to {\rm Lie}_k$ (see \cite[Theorem 2.1]{am20}), where for any commutative algebra $X$ the tensor product $\mathfrak{h} \ot X$ is endowed with the current Lie algebra structure. In order to avoid dealing with cumbersome notation, when there is no fear of confusion, we denote $\mathcal{A} = {\mathcal A}(\mathfrak{h},\, \mathfrak{g})$. Furthermore, if $\mathfrak{h} = \mathfrak{g}$, then the corresponding universal algebra ${\mathcal A} (\mathfrak{h}, \, \mathfrak{h})$ will be denoted simply by ${\mathcal B}$.  The notation is meant to highlight the fact that ${\mathcal B}$ is a bialgebra;  in fact, it admits a unique bialgebra structure such that $\mathfrak{h}$ becomes a right ${\mathcal B}$-comodule with respect to $\eta_{\mathfrak{h}} : \mathfrak{h} \to \mathfrak{h} \ot {\mathcal
B}$ where $\eta \colon 1_{{\rm Lie}_{k}} \to \mathfrak{h} \ot {\mathcal A} (\mathfrak{h},-)$ denotes the unit of the adjunction between ${\mathcal A} (\mathfrak{h},-)$ and $\mathfrak{h} \ot -$. More precisely, the comultiplication and the counit on ${\mathcal B}$ are given for any $i$, $j=1, \cdots, n$ by
\begin{equation} \eqlabel{deltaeps}
\Delta (x_{ij}) = \sum_{s=1}^n \, x_{is} \ot x_{sj} \quad {\rm
and} \quad  \varepsilon (x_{ij}) = \delta_{i, j}1_{k}
\end{equation}

For basic categorical concepts we refer the reader to \cite{mlane}
and for unexplained notions pertaining to Lie and Hopf algebras to \cite{jacobson} and \cite{Sw}, respectively.

\section{Universal modules}\selabel{sect2}

Our first important result provides a way of defining a Lie $ \mathfrak{g}$-module structure on the tensor product between a Lie $ \mathfrak{h}$-module and an $\mathcal{A}$-module. 

\begin{theorem}\thlabel{tensor_f}
Let $(U, \curvearrowright) \in {}_{ \mathfrak{h}}\Ll\Mm$ be a Lie $ \mathfrak{h}$-module and $(V, \cdot) \in {}_{\mathcal{A}} \Mm$ an $\mathcal{A}$-module. 
Then $(U \otimes V, \rightharpoonup) \in {}_{ \mathfrak{g}}\Ll\Mm$ is a Lie $\mathfrak{g}$-module where the action of $ \mathfrak{g}$ on $U \otimes V$ is given for all $i \in I$, $l \in U$ and $t \in V$ by:
\begin{equation}\eqlabel{0.0.2}
f_{i} \rightharpoonup (l \ot t) = \sum_{j=1}^{n}\, (e_{j} \curvearrowright  l) \ot (x_{ji} \cdot t)
\end{equation}
\end{theorem}
\begin{proof}
Indeed, having in mind that $(U, \curvearrowright)$ is a Lie module and 
${\mathcal A} = {\mathcal A}( \mathfrak{h},\, \mathfrak{g})$ is a commutative algebra, we have:
\begin{eqnarray*}
&& [f_{i},\, f_{j}]\rightharpoonup (l \ot t) \stackrel{\equref{const3.4}} = \sum_{u \in B_{i,j}} \beta^{u}_{i,j}\, f_{u} \rightharpoonup (l \ot t) \stackrel{\equref{0.0.2}} = \sum_{u \in B_{i,j}}\, \sum_{r=1}^{n}  \beta^{u}_{i,j}\, (e_{r} \curvearrowright  l) \ot (x_{ru} \cdot t)\\
&& = \sum_{r=1}^{n}\, (e_{r} \curvearrowright  l) \ot \bigl(\sum_{u \in B_{i,j}}\, \beta^{u}_{i,j} \,x_{ru}\bigl) \cdot t  \stackrel{\equref{relatii}} = \sum_{s,p,r=1}^{n} \, \tau_{s, p}^r \, (e_{r} \curvearrowright  l) \ot (x_{si} x_{pj})\cdot t\\
&& = \sum_{s,p=1}^{n} \, \bigl(\sum_{r=1}^{n}\, \tau_{s, p}^r \, e_{r}\bigl) \curvearrowright  l \ot (x_{si} x_{pj})\cdot t \stackrel{\equref{const1}}= \sum_{s,p=1}^{n} \, [e_{s},\, e_{p}] \curvearrowright  l \ot (x_{si} x_{pj})\cdot t \\
&& = \sum_{s,p=1}^{n} \, e_{s} \curvearrowright (e_{p} \curvearrowright l)\, \ot \, x_{si}\cdot (x_{pj} \cdot t) - \sum_{s,p=1}^{n} \, e_{p} \curvearrowright (e_{s} \curvearrowright l)\, \ot \, x_{pj}\cdot (x_{si} \cdot t)
\end{eqnarray*}
\begin{eqnarray*}
&& \stackrel{\equref{0.0.2}}=  f_{i} \rightharpoonup \sum_{p=1}^{n}\,  (e_{p} \curvearrowright l)\, \ot \, (x_{pj} \cdot t) -  f_{j} \rightharpoonup \sum_{s=1}^{n}\,  (e_{s} \curvearrowright l)\, \ot \, (x_{si} \cdot t) \\
&& \stackrel{\equref{0.0.2}}=  f_{i} \rightharpoonup \bigl(f_{j} \rightharpoonup (l \ot t)\bigl) \, - \, f_{j} \rightharpoonup \bigl(f_{i} \rightharpoonup (l \ot t)\bigl)
\end{eqnarray*}
for all $i$, $j \in I$ and $l \in U$, $t \in V$, i.e. $(U \otimes V,\, \rightharpoonup)$ is a left Lie $\mathfrak{g}$-module. 
\end{proof}

Inspired by \thref{tensor_f} we can consider two types of universal modules. 

\subsection{The universal $\mathcal{A}$-module}
The first such universal module is associated with a Lie $\mathfrak{h}$-module and a Lie $\mathfrak{g}$-module as follows:

\begin{definition}\delabel{u_1}
Given a Lie $\mathfrak{h}$-module $U$ and  a Lie $\mathfrak{g}$-module $Z$, the \emph{universal $\mathcal{A}$-module of $U$ and $Z$} is a pair $\bigl({\mathcal U} (U, \, Z), \rho_{{\mathcal U} (U, \, Z)}\bigl)$ consisting of an $\mathcal{A}$-module ${\mathcal U} (U, \, Z)$ and a morphism of Lie $\mathfrak{g}$-modules $\rho_{{\mathcal U} (U, \, Z)} \colon Z \to U \ot {\mathcal U} (U, \, Z)$ such that for any other pair $(X,f)$ consisting of an $\mathcal{A}$-module $X$ and a morphism of Lie $\mathfrak{g}$-modules $f \colon Z \to U \ot X$, there exists a unique morphism of $\mathcal{A}$-modules $g \colon {\mathcal U} (U, \, Z) \to X$ such that the following diagram is commutative:
\begin{eqnarray}\eqlabel{uni1}
\xymatrix {& Z \ar[rr]^-{\rho_{{\mathcal U} (U, \, Z)}} \ar[drr]_{f}
& {} & {U \ot {\mathcal U} (U, \, Z)} \ar[d]^{{\rm Id}_{U} \ot g}\\
& {} & {} & {U \ot X}}
\end{eqnarray}
\end{definition}

In other words, the above definition is saying that, when it exists, the universal $\mathcal{A}$-module of $U$ and $Z$ is in fact the initial object of the category whose objects are pairs $(X,f)$ consisting of an $\mathcal{A}$-module $X$ and a morphism of Lie $\mathfrak{g}$-modules $f \colon Z \to U \ot X$, while morphisms between two such objects $(X,f)$ and $(X',f')$ are defined to be $\mathcal{A}$-module maps $g \colon X \to X'$ satisfying $({\rm Id}_{U} \ot g) \circ f = f'$.

As direct consequences of the above definition, we obtain the following:

\begin{corollary}\colabel{c_1}
Let $U$ be a Lie $\mathfrak{h}$-module. Then, for all Lie $\mathfrak{g}$-modules $Z$ and all $\mathcal{A}$-modules $X$, we have a bijective correspondence between: 
\begin{enumerate}
\item[(1)] Lie $\mathfrak{g}$-module maps $f \colon Z \to U \ot X$;
\item[(2)] $\mathcal{A}$-module maps $g \colon {\mathcal U}(U,\, Z) \to X$.
\end{enumerate}
\end{corollary}

Under the appropriate finite-dimensionality assumptions required for all Manin-Tambara type constructions, the universal $\mathcal{A}$-module introduced in \deref{u_1} exists:

\begin{theorem}\thlabel{exist1}
If $U$ is a finite dimensional Lie $\mathfrak{h}$-module then the universal $\mathcal{A}$ -module of $U$ and any other Lie $\mathfrak{g}$-module $Z$ exists. 
\end{theorem}
\begin{proof}
Let $\{u_1, \cdots, u_m\}$, $m \in \NN^{*}$, be a $k$-basis of the Lie module $U$ and denote by $\omega^{t}_{ij} \in k$ the structure constants of $U$ with respect to its Lie $\mathfrak{h}$-module structure $ \curvearrowright$, i.e.  for all $i = 1,
\cdots, n$, $j = 1, \cdots, m$ we have:
\begin{equation}\eqlabel{const0.1}
e_i\, \curvearrowright\, u_{j} = \sum_{s=1}^m \,
\omega_{i, j}^s \, u_s
\end{equation}
Furthermore, consider $\{z_{r} ~|~ r \in J\}$ to be a $k$-basis for the arbitrary Lie ${\mathfrak{g}}$-module $Z$ and if $\looparrowright$ denotes its Lie module structure, then for all $j \in I$ and $r \in J$ we can find a finite subset $T_{j,r}$ of $J$ such that:
\begin{equation}\eqlabel{const0.2}
f_j\, \looparrowright\, z_{r} = \sum_{l \in T_{j,r}} \,
\eta_{j, r}^l \, z_l
\end{equation}
where $\eta_{j, r}^l \in k$ for all $j \in I$, $r \in J$, and $l \in T_{j,r}$.

Consider now ${\mathcal T}(U, Z)$ to be the free $\mathcal{A}$-module on the set $\{Y_{i j} ~|~ i = 1, \cdots, m,\, j \in J\}$ and denote by ${\mathcal U}(U, Z)$ the quotient of ${\mathcal T}(U, Z)$ by its $\mathcal{A}$-submodule generated by the following elements:
\begin{eqnarray}
\sum_{p \in
T_{j, i}} \, \eta_{j, i}^p \, Y_{sp} - \sum_{t = 1}^{m}\sum_{r= 1}^{n}\,
\omega_{r, t}^s \, x_{r j} \bullet Y_{t i}\eqlabel{0.4}
\end{eqnarray}
for all $s = 1, \cdots, m$, $i \in J$ and $j \in I$, where $\bullet$ denotes the $\mathcal{A}$-module action on ${\mathcal T}(U, Z)$.

Denoting $y_{tj} := \widehat{Y_{tj}}$, where $\widehat{Y_{tj}}$ stands for the equivalence class of ${Y_{tj}}$ in the quotient
module ${\mathcal U} (U, \, Z)$, it follows that the relations below hold in ${\mathcal U} (U, \, Z)$:
\begin{eqnarray}
\sum_{p \in
T_{j, i}} \, \eta_{j, i}^p \, y_{sp} = \sum_{t = 1}^{m}\sum_{r= 1}^{n}\,
\omega_{r, t}^s \, x_{r j} \bullet y_{t i} \eqlabel{0.6}
\end{eqnarray}
for all $s = 1, \cdots, m$, $i \in J$ and $j \in I$.

Furthermore, we can define a morphism of Lie $\mathfrak{g}$-modules $\rho_{{\mathcal U} (U, \, Z)} \colon Z \to U \ot {\mathcal U} (U, \, Z)$ as follows:
\begin{equation}\eqlabel{0.7}
\rho_{{\mathcal U} (U, \, Z)}(z_r) := \sum_{s=1}^m \, u_s \ot y_{sr},\quad {\rm for\,\, all}\,\, r\in J.
\end{equation}
It follows now that for all $j \in I$ and $i \in J$ we have:
\begin{eqnarray*}
&&\hspace*{-7mm} \rho_{{\mathcal U} (U, \, Z)}(f_{j} \looparrowright \, z_{i}) \stackrel{\equref{const0.2}}= \rho_{{\mathcal U}(U, Z)}\Bigl(\sum_{p \in T_{j,i}} \eta_{ji}^{p}\, z_{p} \Bigl) = \sum_{p \in T_{j,i}}\sum_{s = 1}^{m}\, \eta_{ji}^{p}\,u_{s} \ot y_{sp} = \sum_{s=1}^{m}\Bigl(u_{s} \ot \sum_{p \in T_{j,i}} \eta_{ji}^{p}\, y_{sp}\Bigl)\\
&&\hspace*{-7mm}\stackrel{\equref{0.6}} =  \sum_{s, t = 1}^{m}\sum_{r= 1}^{n}\, \omega_{r, t}^s\, u_{s} \ot
  x_{r j} \bullet y_{t i} =  \sum_{t = 1}^{m}\sum_{r= 1}^{n}\, \Bigl(\sum_{s=1}^{m}\omega_{r, t}^s\, u_{s}\Bigl) \ot
  x_{r j} \bullet y_{t i}\\
&&\hspace*{-7mm}\stackrel{\equref{const0.1}} =  \sum_{t=1}^{m} \sum_{r=1}^{n}  e_r \, \curvearrowright \, u_{t} \ot
  x_{r j} \bullet y_{t i} \stackrel{\equref{0.0.2}}  = \sum_{t=1}^{m} f_{j} \rightharpoonup (u_{t} \ot y_{ti}) =  f_{j} \rightharpoonup \sum_{t=1}^{m} u_{t} \ot y_{ti}\\
&&\hspace*{-7mm} \stackrel{\equref{0.7}}= f_{j} \rightharpoonup \rho_{{\mathcal U} (U, \, Z)}(z_{i})
\end{eqnarray*}
which shows that $\rho_{{\mathcal U} (U, \, Z)}$ is indeed a Lie $\mathfrak{g}$-modules map.

We will show that the pair $\bigl({\mathcal U} (U, \, Z), \rho_{{\mathcal U} (U, \, Z)}\bigl)$ constructed above is in fact the universal $\mathcal{A}$-module of $U$ and $Z$. To this end, consider a pair $(X,f)$ consisting of an $\mathcal{A}$-module $X$ and a morphism of Lie $\mathfrak{g}$-modules $f \colon Z \to U \ot X$. Let $\{w_{sr} \, | \, s = 1, \cdots, m, r\in J \}$ be a family of elements of $X$ such that for
all $r\in J$ we have:
\begin{equation}\eqlabel{constfmor}
f(z_r) = \sum_{s=1}^m \, u_s \ot w_{sr}
\end{equation}
Furthermore, as $f \colon Z \to U\ot X$ is a Lie $\mathfrak{g}$-modules
map, a straightforward computation shows that the following compatibilities hold for all $s = 1, \cdots, m$, $i \in J$ and $j \in I$:
\begin{eqnarray}
\sum_{p \in
T_{j, i}} \, \eta_{j, i}^p \, w_{sp} = \sum_{t = 1}^{m}\sum_{r= 1}^{n}\,
\omega_{r, t}^s \, x_{r j} \cdot w_{t i} \eqlabel{rel0.0.4}
\end{eqnarray}
where $\cdot$ denotes the $\mathcal{A}$-module action on $X$. 

The universal property of the free module yields a unique $\mathcal{A}$-module map $\overline{g} \colon {\mathcal T}(U, Z) \to X$ such that $\overline{g}(Y_{sr}) =
w_{sr}$, for all $s = 1, \cdots, m$ and $r\in J$. Moreover, ${\rm Ker} (\overline{g})$ contains the $\mathcal{A}$-submodule of ${\mathcal T}(U, Z)$ generated by the elements listed in \equref{0.4}. Indeed, as $\overline{g} \colon {\mathcal U}(U, Z) \to X$ is a morphism of $\mathcal{A}$-modules we have: 
\begin{eqnarray*}
\overline{g}\Bigl(\sum_{p \in
T_{j, i}} \, \eta_{j, i}^p \, Y_{sp} - \sum_{t = 1}^{m}\sum_{r= 1}^{n}\,
\omega_{r, t}^s \, x_{r j} \bullet Y_{t i}\Bigl) = \sum_{p \in
T_{j, i}} \, \eta_{j, i}^p \, w_{sp} - \sum_{t = 1}^{m}\sum_{r= 1}^{n}\,
\omega_{r, t}^s \, x_{r j} \cdot w_{t i} \stackrel{\equref{rel0.0.4}} = 0
\end{eqnarray*}
for all $s = 1, \cdots, m$, $i \in J$ and $j \in I$. This shows that there exists a unique $\mathcal{A}$-modules map $g \colon {\mathcal U}(U, Z) \to X$ such that $g(y_{sr}) = z_{sr}$, for all $s = 1, \cdots, m$ and $r\in J$. This implies that for all $r \in J$ we have:
\begin{eqnarray*}
\bigl({\rm Id}_{U} \ot g \bigl) \circ \,
\rho_{{\mathcal U} (U, \, Z)} (z_r) = \bigl({\rm Id}_{U} \ot
g \bigl) \bigl( \sum_{s=1}^{m} \, u_s \ot y_{sr} \bigl) = \sum_{s=1}^{m} \, u_s \ot w_{sr} \stackrel{\equref{constfmor}}
= f (z_{r})
\end{eqnarray*}
which means precisely that diagram \equref{uni1} is commutative. Moreover, $g$ is obviously the unique $\mathcal{A}$-modules map with this property and the proof is now finished. 
\end{proof}

\subsection*{The case $\mathfrak{g} = \mathfrak{h}$}

Particularizing the results of \seref{sect2} for $\mathfrak{g} = \mathfrak{h}$, where $\mathfrak{h}$ is the finite dimensional Lie algebra defined in \equref{const1}, leads to the following interesting consequences. According to the discussion in Preliminaries, the universal algebra $\mathcal{A}(\mathfrak{h}, \mathfrak{h})$ denoted by ${\mathcal B}$ is in this case a bialgebra with coalgebra structure depicted in \equref{deltaeps}. This allows us to see the tensor product ${\mathcal U}(U,Z) \ot {\mathcal U}(U,Z)$ as well as the base field $k$ as ${\mathcal B}$-modules via the comultiplication and the counit of ${\mathcal B}$ as follows:
\begin{eqnarray}
&& x_{ij}* (y \ot t) = \sum_{t=1}^{n} x_{it} \bullet y \ot x_{tj} \bullet t\eqlabel{mod011}\\
&& x_{ij} \cdot \alpha = \delta_{ij}\alpha \eqlabel{mod022}
\end{eqnarray}
for all $x_{ij} \in {\mathcal B}$, $y$, $t \in {\mathcal U}(U,Z)$ and $\alpha \in k$, where $\bullet$ denotes the ${\mathcal B}$-module structure on ${\mathcal U}(U,Z)$ as in the proof of \thref{exist1}.  

First we show that if $U$ is a finite dimensional Lie $\mathfrak{h}$-module as considered in \equref{const0.1}, then the ${\mathcal B}$-module ${\mathcal U}(U,\,U)$ denoted by ${\mathcal U}(U)$ admits a coalgebra structure with respect to which $\bigl(U, \rho_{{\mathcal U}(U)}\bigl)$ becomes a right ${\mathcal U}(U)$-comodule.

\begin{proposition}\prlabel{p2}
Let $U$ be a finite dimensional Lie $\mathfrak{h}$-module. There exists a unique coalgebra structure on ${\mathcal U}(U)$ such that $\bigl(U, \rho_{{\mathcal U}(U)}\bigl)$ becomes a right ${\mathcal U}(U)$-comodule.
\end{proposition}
\begin{proof}
In particular both ${\mathcal U}(U) \ot {\mathcal U}(U)$ and $k$ are ${\mathcal B}$-modules via the formulas \equref{mod011} and \equref{mod022} respectively. Therefore, $U \ot {\mathcal U}(U) \ot {\mathcal U}(U)$ and $U \ot k$ are Lie $\mathfrak{h}$-modules via \equref{0.0.2}. Furthermore, it can be easily checked that the maps $\bigl(\rho_{{\mathcal U}(U)} \ot {\rm Id}_{{\mathcal U}(U)}\bigl) \circ \rho_{{\mathcal U}(U)} \colon U \to U \ot {\mathcal U}(U) \ot {\mathcal U}(U)$ and ${\rm can}_{U} \colon U \to U \ot k$ are morphisms of Lie $\mathfrak{h}$-modules, where ${\rm can}_{U} \colon U \to U \ot k$ is the canonical isomorphism. Now \deref{u_1} yields a unique ${\mathcal B}$-modules map $\Delta \colon {\mathcal U}(U) \to {\mathcal U}(U) \ot {\mathcal U}(U)$ such that the following diagram is commutative:
\begin{eqnarray*}
\xymatrix {& U \ar[r]^-{\rho_{{\mathcal U}(U)}} \ar[ddr]_-{\bigl(\rho_{{\mathcal U}(U)} \ot {\rm Id}_{{\mathcal U}(U)}\bigl) \circ \rho_{{\mathcal U}(U)}}
& {U \ot {\mathcal U}(U)} \ar[dd]^{{\rm Id}_{U} \ot \Delta}\\
& {} & {} \\
& {} & {U \ot {\mathcal U}(U) \ot {\mathcal U}(U)}}
\end{eqnarray*}
Similarly, we obtain a unique ${\mathcal B}$-modules map $\varepsilon \colon  {\mathcal U}(U) \to k$ such that the following diagram is commutative:
\begin{eqnarray*}
\xymatrix {& U \ar[r]^-{\rho_{{\mathcal U}(U)}} \ar[dr]_{{\rm can}_{U}}
& {U \ot {\mathcal U}(U)} \ar[d]^{{\rm Id}_{U} \ot \varepsilon}\\
& {} & {U \ot k}}
\end{eqnarray*}
A straightforward computation shows that the commutativity of the two diagrams above implies that $\Delta$ and $\varepsilon$ take the following form for all $l$, $t =1,\cdots, m$:
\begin{eqnarray*}
\Delta(y_{lt}) = \sum_{s=1}^{m}\, y_{ls} \ot y_{st},\quad \varepsilon(y_{lt}) = \delta_{lt}1_{k}.
\end{eqnarray*}
It is now obvious that $\bigl({\mathcal U}(U), \Delta, \varepsilon\bigl)$ form a coalgebra. Finally, by the commutativity of the two diagrams above we obtain that $\bigl(U, \rho_{{\mathcal U}(U)}\bigl)$ is a right ${\mathcal U}(U)$-comodule.  
\end{proof}

\begin{remark}
It is worth pointing out that with the coalgebra structure introduced above, ${\mathcal U}(U)$ becomes a ${\mathcal B}$-module coalgebra. 
Indeed, having in mind that both $\Delta$ and $\varepsilon$ are ${\mathcal B}$-module maps, we have:
\begin{eqnarray*}
&&\Delta(x_{ab} \bullet y_{lt}) = x_{ab}* \Delta(y_{lt}) = x_{ab} * \bigl(\sum_{s=1}^{m}\, y_{ls} \ot y_{st}\bigl) \stackrel{\equref{mod011}} = \sum_{c=1}^{n}\sum_{s=1}^{m}\, x_{ac} \bullet y_{ls} \ot x_{cb} \bullet y_{st}\\
&& = (x_{ab})_{(1)} \bullet (y_{lt})_{(1)} \ot (x_{ab})_{(2)} \bullet (y_{lt})_{(2)}
\end{eqnarray*}
and
\begin{eqnarray*}
&&\varepsilon(x_{ab} \bullet y_{lt}) = x_{ab}\cdot \varepsilon(y_{lt})  \stackrel{\equref{mod022}} = \delta_{ab}\, \varepsilon(y_{lt}) =\varepsilon(x_{ab}) \, \varepsilon(y_{lt}). 
\end{eqnarray*}

This shows that $\bullet$ is a coalgebra map, as desired.
\end{remark}

It turns out that the pair $\bigl({\mathcal U}(U),\, \rho_{{\mathcal U}(U)}\bigl)$ is universal in the following way:

\begin{proposition}
For any coalgebra $X$ with a ${\mathcal B}$-module structure and any Lie $\mathfrak{h}$-module morphism $\psi \colon U \to U\ot X$ which makes $U$ into a right $X$-comodule, there exists a unique ${\mathcal B}$-modules and coalgebra morphism $\theta \colon {\mathcal U}(U) \to X$ such that the following diagram is commutative:
\begin{eqnarray*}
\xymatrix {& U \ar[r]^-{ \rho_{{\mathcal U}(U)}} \ar[dr]_{\psi}
& {U \ot {\mathcal U}(U)} \ar[d]^{{\rm Id}_{U} \ot \theta}\\
& {} & {U \ot X}}
\end{eqnarray*}
\end{proposition}  
\begin{proof}
In light of \deref{u_1}, such a unique ${\mathcal A}$-modules map $\theta$ exists. We are left to show that $\theta$ is also a coalgebra map. From the proof of \thref{exist1} we know that $\theta$ is defined for all $l$, $t = 1, \cdots, m$ by $\theta(y_{lt}) = z_{lt}$ where $z_{lt}$ are elements of $X$ such that for all $r = 1, \cdots, m$ we have $\psi(u_{r}) = \sum_{s=1}^{m}\, u_{s} \ot z_{sr}$. As $(U,\, \psi)$ is a right comodule, we obtain:
\begin{eqnarray*}
\Delta(z_{lt}) = \sum_{s=1}^{m}\, z_{ls} \ot z_{st},\quad \varepsilon(z_{lt}) = \delta_{lt}1_{k}.
\end{eqnarray*}
To this end, we have:
\begin{eqnarray*}
\Delta\bigl(\theta(y_{lt})\bigl) = \Delta(z_{lt}) = \sum_{s=1}^{m} \, z_{ls} \ot z_{st} = \sum_{s=1}^{m} \,\theta(y_{ls}) \ot \theta(y_{st}) = (\theta \ot \theta)\circ \Delta(y_{lt})
\end{eqnarray*}
Similarly one can check that $\varepsilon \circ \theta = \varepsilon$ which shows that $\theta$ is indeed a coalgebra map.
\end{proof}

\subsection{The universal $\mathfrak{h}$-module}

The second type of universal module one can consider is the following:

\begin{definition}\delabel{u_2}
Given an $\mathcal{A}$-module $V$ and a Lie $\mathfrak{g}$-module $W$, the \emph{universal Lie $\mathfrak{h}$-module of $V$ and $W$} is a pair $\bigl({\mathcal V} (V, \, W), \tau_{{\mathcal V} (V, \, W)}\bigl)$ consisting of a Lie $\mathfrak{h}$-module ${\mathcal V} (V, \, W)$ and a morphism of Lie $\mathfrak{g}$-modules $\tau_{{\mathcal V} (V, \, W)} \colon W \to {\mathcal V} (V, \, W) \ot V$ such that for any other pair $(Y,f)$ consisting of a Lie $\mathfrak{h}$-module $Y$ and a morphism of Lie $\mathfrak{g}$-modules $f \colon W \to Y \ot V$, there exists a unique morphism of Lie $\mathfrak{h}$-modules $g \colon {\mathcal V} (V, \, W) \to Y$ such that the following diagram is commutative:
\begin{eqnarray}\eqlabel{uni2}
\xymatrix {& W \ar[rr]^-{\tau_{{\mathcal V} (V, \, W)}} \ar[drr]_{f}
& {} & {{\mathcal V} (V, \, W) \ot V} \ar[d]^{g \ot {\rm Id}_{V}}\\
& {} &  {} &{Y \ot V}}
\end{eqnarray}
\end{definition}
The universal Lie $\mathfrak{h}$-module of $V$ and $W$, when it exists, can again be seen as the initial object of the category whose objects are pairs $(Y,f)$ consisting of a Lie $\mathfrak{h}$-module $Y$ and a morphism of Lie $\mathfrak{g}$-modules $f \colon W \to Y \ot V$, while morphisms between two such objects $(Y,f)$ and $(Y',f')$ are defined to be Lie $\mathfrak{h}$-module maps $g \colon Y \to Y'$ satisfying $(g \ot {\rm Id}_{V}) \circ f = f'$.

\begin{corollary}\colabel{c_2}
Let $V$ be an $\mathcal{A}$-module. Then, for all  Lie $\mathfrak{g}$-modules $W$ and all Lie $\mathfrak{h}$-modules $Y$, we have a bijective correspondence between:
\begin{enumerate}
\item[(1)] Lie $\mathfrak{g}$-module maps $f \colon W \to Y \ot V$;
\item[(2)] Lie $\mathfrak{h}$-module maps $g \colon {\mathcal V}(V,\, W) \to Y$.
\end{enumerate}
\end{corollary}

The universal $\mathfrak{h}$-module introduced in \deref{u_2} also exists provided that the $\mathcal{A}$-module $V$ is finite dimensional.

\begin{theorem}\thlabel{exist2}
If $V$ is a finite dimensional $\mathcal{A}$-module then the universal Lie $\mathfrak{h}$-module of $V$ and any other Lie $\mathfrak{g}$-module $W$ exists. 
\end{theorem}
\begin{proof}
As this proof is somewhat similar in spirit with the one of \thref{exist1}, we will be brief and provide only the main ingredients required for the construction of the universal Lie $\mathfrak{h}$-module.

Let $\{v_1, \cdots, v_l\}$, $l \in \NN^{*}$, be a $k$-basis of the finite dimensional $\mathcal{A}$-module $V$ and denote by $\gamma^{t}_{r,i,j} \in k$ the structure constants of $V$ with respect to its $\mathcal{A}$-module structure $\cdot$, i.e.  for all $r = 1,
\cdots, n$, $i \in I$ and $j = 1, \cdots, l$ we have:
\begin{equation}\eqlabel{const0.1.1}
x_{r i} \,  \cdot \, v_{j}= \sum_{s=1}^l \,
\gamma_{r, i, j}^s \, v_s
\end{equation}
Consider $\{w_{r} ~|~ r \in T\}$  to be a $k$-basis for $W$ and if $\triangleright$ denotes its Lie $\mathfrak{g}$-module structure, then for all $j \in I$ and $r \in T$ we can find a finite subset $S_{j,r}$ of $T$ such that $f_j\, \triangleright\, w_{r} = \sum_{p \in S_{j,r}} \,
\sigma_{j, r}^p \, w_p$, where $\sigma_{j, r}^p \in k$ for all $j \in I$, $r \in T$, and $p \in S_{j,r}$.

Now let ${\mathcal S}(V, W)$ be the free Lie $\mathfrak{h}$-module on the set $\{Y_{r i} ~|~ r \in T,\, i = 1, \cdots, l\}$ and denote by ${\mathcal V}(V, W)$ the quotient of ${\mathcal S}(V, W)$ by its Lie $\mathfrak{h}$-submodule generated by the following elements:
\begin{eqnarray*}
\sum_{p \in
S_{j, r}} \, \sigma_{j, r}^p \, Y_{ps} - \sum_{k = 1}^{l}\sum_{t= 1}^{n}\,
\gamma_{t,j,k}^s \, e_{t} \blacktriangleright Y_{rk} \eqlabel{0.4.4}
\end{eqnarray*}
for all $s = 1, \cdots, l$, $r \in T$ and $j \in I$, where $\blacktriangleright$ denotes the $\mathfrak{h}$-module action on ${\mathcal S}(V, W)$.

By denoting $y_{r i} := \widehat{Y_{r i}}$, where $\widehat{Y_{r i}}$ stands for the equivalence class of ${Y_{r i}}$ in the quotient
module ${\mathcal V}(V, W)$, it follows that the relations below hold in ${\mathcal V}(V, W)$:
\begin{eqnarray*}
\sum_{p \in
S_{j, r}} \, \sigma_{j, r}^p \, y_{ps} = \sum_{k = 1}^{l}\sum_{t= 1}^{n}\,
\gamma_{t,j,k}^s \, e_{t} \blacktriangleright y_{rk}  \eqlabel{0.6.6}
\end{eqnarray*}
for all $s = 1, \cdots, l$, $r \in T$ and $j \in I$.

It can now be easily seen, as in the proof  of \thref{exist1}, that the pair $({\mathcal V} (V, \, W),\, \tau_{{\mathcal V} (V, \, W)})$ is the universal Lie $\mathfrak{h}$-module of $V$ and $W$, where $\tau_{{\mathcal V} (V, \, W)} \colon W \to{\mathcal V} (V, \, W) \ot V$ is the morphism of Lie $\mathfrak{g}$-modules defined for all $r\in T$ as follows:
\begin{eqnarray*}
\tau_{{\mathcal V} (V, \, W)}(w_r) := \sum_{s=1}^l \, y_{rs} \ot v_{s}.
\end{eqnarray*}
\end{proof}

\section{Functors between module categories}\selabel{sect3}

In this section we show that the two universal module constructions previously introduced are functorial and, moreover, if certain conditions are fulfilled the corresponding functors admit right adjoints. We start, however, by stating the following easy consequence of  \thref{tensor_f}:

\begin{proposition}\prlabel{tens_fun}
Let $(U, \curvearrowright) \in {}_{ \mathfrak{h}}\Ll\Mm$ and $(V, \cdot) \in {}_{\mathcal A} \Mm$. Then:
\begin{enumerate}
\item[1)] We have a functor $U \ot - \colon {}_{\mathcal{A}} \Mm \to {}_{\mathfrak{g}}{\mathcal L} \Mm$ from the category of ${\mathcal A}$-modules to the category of Lie $\mathfrak{g}$-modules;\\
\item[2)] We have a functor $ - \ot V \colon {}_{\mathfrak{h}}{\mathcal L} \Mm \to {}_{\mathfrak{g}}{\mathcal L} \Mm$ between the categories of Lie modules over $\mathfrak{h}$ and $\mathfrak{g}$ respectively. 
\end{enumerate}
\end{proposition}
\begin{proof}
In light of \thref{tensor_f}, we are only left to show that morphisms behave well with respect to the corresponding associative or Lie module structures. We will treat only the first statement and leave the second one to the reader. To this end, consider $(V,\, \cdot)$ and $(V',\, \bullet)$ two $\mathcal{A}$-modules, $\rightharpoonup$ and $\rightharpoonup'$ the corresponding induced Lie $\mathfrak{g}$-module actions via \equref{0.0.2} and $g \colon V \to V'$ a morphism in ${}_{\mathcal{A}} \Mm$ . Then, for all $i \in I$, $l \in U$ and $t \in V$ we have:
\begin{eqnarray*}
&&({\rm Id}_{U} \ot g)\bigl(f_{i} \rightharpoonup (l \ot t)\bigl) \stackrel{\equref{0.0.2}} = \sum_{j=1}^{n}\, (e_{j} \curvearrowright  l) \ot g(x_{ji} \cdot t) = \sum_{j=1}^{n}\, (e_{j} \curvearrowright  l) \ot x_{ji} \bullet g(t)\\
&& \stackrel{\equref{0.0.2}} = f_{i} \rightharpoonup' \bigl(l \ot g(t)\bigl)
\end{eqnarray*}
\end{proof}

\begin{example}
Let $\mathfrak{h} = k$, the $1$-dimensional Lie algebra, and $\mathfrak{g}$ an arbitrary Lie algebra as considered in \equref{const3.4}. Recall from \cite[Examples 2.5, (3)]{am20} that in this case we have $\mathcal{A} = S (\mathfrak{g}/\mathfrak{g}')$, the symmetric algebra of $\mathfrak{g}/\mathfrak{g}'$, where $\mathfrak{g}'$ is the derived subalgebra of $\mathfrak{g}$. Equivalently, we can describe $\mathcal{A}$ as the polynomial algebra on variables $x_{t}$, with $t \in I$, subject to $\sum_{u \in B_{i, j}} \, \beta_{i, j}^u \, x_{u} = 0$, for all $i$, $j \in I$.

Now \prref{tens_fun} yields a fully faithful functor $F \colon {}_{\mathcal{A}}\Mm \to {}_{\mathfrak{g}}{\mathcal L} \Mm$ defined as follows for all $\mathcal{A}$-modules $(V,\, \star)$ and all $\mathcal{A}$-module maps $f$:
\begin{eqnarray*}
F\bigl((V,\, \star)\bigl) = \bigl(V,\, \rightharpoonup^{\star}\bigl), \quad F(f) = f,
\end{eqnarray*}
where $\bigl(V,\, \rightharpoonup^{\star}\bigl)$ is the Lie $\mathfrak{g}$-module defined for all $t \in I$ and $v \in V$ by 
\begin{equation}\eqlabel{ind19}
f_{t} \rightharpoonup^{\star} v = x_{t} \star v
\end{equation}
 $F$ is obviously a faithful functor. Furthermore, let $h \colon \bigl(U,\, \rightharpoonup^{*}\bigl) \to  \bigl(V,\, \rightharpoonup^{\star}\bigl)$ be a morphism in ${}_{\mathfrak{g}}{\mathcal L} \Mm$, where $\rightharpoonup^{*}$ and $\rightharpoonup^{\star}$ are Lie $\mathfrak{g}$-modules induced from $\mathcal{A}$-module structures $(U,\, *)$ and $(V,\, \star)$ respectively as in \equref{ind19}. Then, for all $t \in I$ and $u \in U$, we have $h\bigl(f_{t}\rightharpoonup^{*} u \bigl) = f_{t} \rightharpoonup^{\star} h(u)$ which implies $h(x_{t} * u) = x_{t} \star h(u)$ and therefore $h \colon (U,\, *) \to (V,\, \star)$ is also a morphism in ${}_{\mathcal{A}}\Mm$. This shows that $F$ is also a full functor. 
\end{example}
\newpage
We consider now the universal module functors:

\begin{theorem}\thlabel{functorial}
Let $U$ be a finite dimensional Lie $\mathfrak{h}$-module and $V$ a finite dimensional $\mathcal{A}$-module. 
\begin{enumerate}
\item[(1)] There exists a functor ${\mathcal U}_{U} \colon {}_{\mathfrak{g}}{\mathcal L} \Mm \to {}_{\mathcal A} \Mm$ defined as follows for all Lie $\mathfrak{g}$-modules $X$, $Y$ and all morphisms $f \colon X \to Y$ in ${}_{\mathfrak{g}}{\mathcal L} \Mm$:
\begin{eqnarray*}
{\mathcal U}_{U}(X) = {\mathcal U} (U, \,X), \qquad {\mathcal U}_{U}(f) = \overline{f}
\end{eqnarray*}
where $\overline{f}\colon {\mathcal U} (U, \, X) \to {\mathcal U} (U, \, Y)$ is the unique $\mathcal{A}$-modules morphism which makes the following diagram commutative:
 \begin{eqnarray}\eqlabel{fun_11}
\xymatrix {& X \ar[rr]^-{\rho_{{\mathcal U} (U, \, X)}} \ar[drr]_-{\rho_{{\mathcal U} (U, \, Y)} \circ f}
& {} & {U \ot {\mathcal U} (U, \, X)} \ar[d]^{{\rm Id}_{U} \ot \overline{f}}\\
& {} & {} & {U \ot {\mathcal U} (U, \, Y)}}
\end{eqnarray}
\item[(2)] There exists a functor ${\mathcal V}_{V} \colon {}_{\mathfrak{g}}{\mathcal L} \Mm \to {}_{\mathfrak{h}}{\mathcal L} \Mm$ defined as follows for all Lie $\mathfrak{g}$-modules $X$, $Y$ and all morphisms $f \colon X \to Y$ in ${}_{\mathfrak{g}}{\mathcal L} \Mm$:
\begin{eqnarray*}
{\mathcal V}_{V}(X) = {\mathcal V} (V, \,X), \qquad {\mathcal V}_{V}(f) = \overline{f}
\end{eqnarray*}
where $\overline{f}\colon {\mathcal V} (V, \, X) \to {\mathcal V} (V, \, Y)$ is the unique morphism of Lie $\mathfrak{h}$-modules which makes the following diagram commutative:
 \begin{eqnarray}\eqlabel{fun_22}
\xymatrix {& X \ar[rr]^-{\tau_{{\mathcal V} (V, \, X)}} \ar[drr]_{\tau_{{\mathcal V} (V, \, Y)} \circ f}
& {} & {{\mathcal V} (V, \, X) \ot V} \ar[d]^{\overline{f} \ot {\rm Id}_{V}}\\
& {} &  {} &{{\mathcal V} (V, \, Y) \ot V}}
\end{eqnarray}
\end{enumerate}
\end{theorem}
\begin{proof}
As the result follows in a straightforward manner by a standard category theory argument, we only sketch the proof of the first assertion. Indeed, if $f = {\rm Id}_{X}$ then ${\rm Id}_{{\mathcal U} (U, \, X)}$ is obviously the unique $\mathcal{A}$-modules morphism which makes diagram \equref{fun_11} commute and therefore ${\mathcal U}_{U}({\rm Id}_{X}) = {\rm Id}_{{\mathcal U} (U, \, X)}$. Moreover, if $f \colon X \to Y$ and $g \colon Y \to W$ are two morphisms in ${}_{\mathfrak{g}}{\mathcal L} \Mm$, then $\overline{g} \circ \overline{f} \colon {\mathcal U} (U, \, X) \to {\mathcal U} (U, \, W)$ is obviously the unique $\mathcal{A}$-modules morphism which makes the following diagram commutative:
 \begin{eqnarray*}
\xymatrix {& Z \ar[rr]^-{\rho_{{\mathcal U} (U, \, X)}} \ar[drr]_-{\rho_{{\mathcal U} (U, \, W)} \circ g \circ f}
& {} & {U \ot {\mathcal U} (U, \, X)} \ar[d]^{{\rm Id}_{U} \ot\bigl( \overline{g}\circ \overline{f}\bigl)}\\
& {} & {} & {U \ot {\mathcal U} (U, \, W)}}
\end{eqnarray*}
and we can conclude that ${\mathcal U}_{U}(g \circ f) = {\mathcal U}_{U}(g) \circ {\mathcal U}_{U}(f)$, as desired.
\end{proof}

Under the appropriate finite-dimensionality assumptions, the functors constructed in \prref{tens_fun} are right adjoints to the universal module functors:

\begin{theorem}\thlabel{adjoint_tens}
Let $(U, \curvearrowright)$ be a finite dimensional Lie $\mathfrak{h}$-module and $(V, \cdot)$ a finite dimensional ${\mathcal A}$-module. Then:
\begin{enumerate}
\item[1)] The following functors form an adjunction:
\begin{eqnarray*}
{\mathcal U}_{U} \colon {}_{\mathfrak{g}}{\mathcal L} \Mm \to {}_{\mathcal A} \Mm,\qquad U \ot - \colon {}_{\mathcal A} \Mm \to {}_{\mathfrak{g}}{\mathcal L} \Mm;
\end{eqnarray*}
\item[2)] Similarly, the following functors also form an adjunction:
\begin{eqnarray*}
{\mathcal V}_{V} \colon {}_{\mathfrak{g}}{\mathcal L} \Mm \to {}_{\mathfrak{h}}{\mathcal L} \Mm,\qquad  - \ot V \colon {}_{\mathfrak{h}}{\mathcal L} \Mm \to {}_{\mathfrak{g}}{\mathcal L} \Mm.
\end{eqnarray*}
\end{enumerate}
\end{theorem}
\begin{proof}
$1)$ As pointed out in \coref{c_1}, for all Lie $\mathfrak{g}$-modules $Z$ and all $A$-modules $X$, there is a bijection between ${\rm Hom}_{{}_A\mathcal{M}} \, \bigl(
{\mathcal U}_{U}(Z), \, X\bigl)$ and ${\rm
Hom}_{\mathfrak{g}{\mathcal L} \Mm} \, (Z, \, U \ot X)$ given as follows for all morphisms of ${\mathcal A}$-modules $\theta \colon {\mathcal U}_{U}(Z) \to X$:
\begin{eqnarray*}
&& \Gamma_{Z, X} \colon {\rm Hom}_{{}_{\mathcal A}\mathcal{M}} \, (
{\mathcal U}_{U}(Z), \, X) \to {\rm
Hom}_{\mathfrak{g}{\mathcal L} \Mm} \, (Z, \, U \ot X), \\
&& \Gamma_{Z, X}(\theta) = ({\rm Id}_{U} \ot \theta) \circ \rho_{{\mathcal U}(U,\,Z)}.
\end{eqnarray*}
The desired conclusion now follows by showing that the above bijection is natural in both variables. This can be easily proved by a straightforward diagram chase and is left to the reader.  

$2)$ Using now \coref{c_2}, for all Lie $\mathfrak{g}$-modules $W$ and all Lie $\mathfrak{h}$-modules $Z$, we obtain a bijection between ${\rm Hom}_{\mathfrak{h}{\mathcal L} \Mm} \, \bigl(
{\mathcal V}_{V} (W), \, Z\bigl)$ and ${\rm
Hom}_{\mathfrak{g}{\mathcal L} \Mm} \, (W, \, Z \ot V)$ defined as follows for all morphisms of Lie $\mathfrak{h}$-modules $\theta \colon {\mathcal V}_{V} (W) \to Z$:
\begin{eqnarray*}
&& \Gamma_{W, Z} \colon {\rm Hom}_{\mathfrak{h}{\mathcal L} \Mm} \, \bigl(
{\mathcal V}_{V} (W), \, Z\bigl) \to {\rm
Hom}_{\mathfrak{g}{\mathcal L} \Mm} \, (W, \, Z \ot V), \\
&& \Gamma_{W, Z}(\theta) = (\theta \ot {\rm Id}_{V}) \circ \rho_{{\mathcal V} (V,\,W)}.
\end{eqnarray*}
\end{proof}

In particular, the two pairs of adjoint functors allow us to travel both ways between the representation categories of the two (arbitrary) Lie algebras $\mathfrak{h}$ and $\mathfrak{g}$ and respectively between the representation category of the associative algebra ${\mathcal A}$ and the representation category of the Lie algebra $\mathfrak{g}$. 

\begin{examples}
1) Let $\rho_{i} \colon \mathfrak{g} \ot W_{i} \to W_{i}$ be Lie $\mathfrak{g}$-modules, where $i=1,\,2$. By the colimit preservation property of left adjoints we can easily conclude that if $U$ is a finite dimensional Lie $\mathfrak{h}$-module then ${\mathcal U}_{U}(W_{1} \oplus W_{2})$ is the direct sum of the ${\mathcal A}$-modules ${\mathcal U}_{U}(W_{1})$ and ${\mathcal U}_{U}(W_{2})$. Similarly, for any finite dimensional ${\mathcal A}$-module $V$, ${\mathcal V}_{V}(W_{1} \oplus W_{2})$ is the direct sum of the Lie $\mathfrak{h}$-modules ${\mathcal V}_{V}(W_{1})$ and ${\mathcal V}_{V}(W_{2})$. This can be easily extended to an arbitrary family of representations. 

2) Let $\CC$ be the field of complex numbers and consider the Lie algebra $\mathfrak{sl}(2, \CC)$ with basis $\{e_1, e_2, e_3\}$ and bracket $\left[e_1, \, e_2 \right] = e_3$, $\left[e_3, \, e_2 \right] = -2
e_2$, $\left[e_3, \, e_1 \right] = 2e_1$. As proved in \cite[Examples 2.9, 2.]{am20} we have ${\mathcal A} = {\mathcal A} (\mathfrak{sl}(2, \CC)) \cong \CC[X_{ij} \,
| \, i, j = 1, 2, 3]/J$, where $ \CC[X_{ij} \,
| \, i, j = 1, 2, 3]$ is the usual polynomial algebra and $J$ is the ideal generated by the
following nine universal polynomials of $\mathfrak{sl}(2, \CC)$:
\begin{eqnarray*}
&& \hspace*{-10mm} X_{13} - 2 X_{12}X_{31} + 2X_{11}X_{32}, \,\,\, X_{11} -
X_{11}X_{33} + X_{13}X_{31},\,\,\,  X_{12} - X_{13}X_{32} +
X_{12}X_{33}\\
&& \hspace*{-10mm} X_{23} - 2 X_{21}X_{32} + 2X_{22}X_{31}, \,\,\, X_{21} -
X_{23}X_{31} + X_{21}X_{33},\,\,\, X_{22} - X_{22}X_{33} +
X_{23}X_{32}\\
&& \hspace*{-10mm} X_{33} - X_{11}X_{22} + X_{12}X_{21}, \,\,\, 2X_{31} - X_{21}X_{13}
+  X_{11}X_{23}, \,\,\, 2X_{32} - X_{12}X_{23} + X_{13}X_{22}.
\end{eqnarray*}
In light of \thref{adjoint_tens}, any finite dimensional irreducible $\mathfrak{sl}(2, \CC)$-module $V_{d}$, where $d \in \NN$ and $V_{d}$ is the subspace of $\CC[X,\,Y]$ consisting of homogeneous polynomials in $X$ and $Y$ of degree $d$, induces a pair of adjoint functors relating the module category over the associative algebra ${\mathcal A}$ to the Lie modules category over $\mathfrak{sl}(2, \CC)$. In particular, if $d_{1}$, $d_{2} \in \NN$ then ${\mathcal U}_{V_{d_{1}}}(V_{d_{2}})$ is an ${\mathcal A}$-module.
\end{examples}

\subsection*{Declaration of competing interest}

There is no competing interest.

\end{document}